\documentclass[12pt, reqno]{amsart}
\usepackage{amsmath, amsthm, amscd, amsfonts, amssymb, graphicx, color}
\usepackage[bookmarksnumbered, colorlinks, plainpages]{hyperref}

\newtheorem{theorem}{Theorem}[section]
\newtheorem{lemma}[theorem]{Lemma}

\newtheorem{corollary}[theorem]{Corollary}
\theoremstyle{definition}

\newtheorem{example}[theorem]{Example}

\theoremstyle{remark}
\newtheorem{remark}[theorem]{Remark}
\numberwithin{equation}{section}

\begin{document}
\setcounter{page}{1}

\title[Generalized derivations on certain Banach algebras]{Generalized derivations on certain Banach algebras}

\author[A. Ebrahimzadeh Esfahani]{  Ali Ebrahimzadeh Esfahani}
\address{Department of Mathematical Sciences,
     Isfahan Uinversity of Technology, Isfahan 84156-83111, Iran;}
\email{\textcolor[rgb]{0.00,0.00,0.84}
{ ali.ebrahimzadeh@math.iut.ac.ir}}

\author[  M. Nemati]{  Mehdi Nemati}
\address{Department of Mathematical Sciences,
	Isfahan Uinversity of Technology,
	Isfahan 84156-83111, Iran;\\
	 \newline
	 and
	 \newline
	 School of Mathematics,
	 Institute for Research in Fundamental Sciences (IPM),
	 P.O. Box: 19395--5746, Tehran, Iran.}
\email{\textcolor[rgb]{0.00,0.00,0.84}{m.nemati@iut.ac.ir}}





\begin{abstract}
Let ${\mathcal A}$ be a Banach algebra with the properties that
$\mathrm{rad}({\mathcal A})={\rm rann}({\mathcal A})$ and the algebra ${\mathcal A}/\mathrm{rad}({\mathcal A})$ is commutative.
We show that a derivation of ${\mathcal A}$ maps ${\mathcal A}$ into  ${\rm rad}({\mathcal A})$. Using this, we determine among other things when a generalized derivation of ${\mathcal A}$ maps ${\mathcal A}$ into ${\rm rad}({\mathcal A})$. We also study $k$-centralizing generalized derivations of ${\mathcal A}$. Then, for a generalized derivation $(\delta, d)$ of ${\mathcal A}$ we obtain a necessary and sufficient
condition for $(\delta^2, d^2)$ to be still a
generalized derivation of ${\mathcal A}$. 
The main applications are concerned with the algebras over locally compact  groups. In particular, we deduce these results for  bidual of Fourier algebras of  discrete amenable groups as an application of our approach.
\end{abstract} \maketitle

\section{Introduction and preliminaries}
 For a Banach algebra ${\mathcal A}$, we denote by ${\rm rad}({\mathcal A})$ the (Jacobson) radical of ${\mathcal A}$ and by ${\rm rann}( {{\mathcal A}})$
the { right annihilator} of ${\mathcal A}$; i.e. the set of all $c\in{\mathcal A}$ such that $ac=0$ for all $a\in{\mathcal A}$. Note that ${\rm rann}({\mathcal A})$ is nilpotant and so ${\rm rann}({\mathcal A})\subseteq\mathrm{rad}({\mathcal A})$; see for example \cite[Proposition 1.5.6]{dales}. A linear map $d: {\mathcal A}\rightarrow {\mathcal A}$ is called a derivation of ${\mathcal A}$ if $d(ab) = d(a)b + ad(b)$ for all $a,b\in{\mathcal A}$.
For each $a\in {\mathcal A}$, the
map $d_a: {\mathcal A}\rightarrow {\mathcal A}$ defined by $d_a(b)=[b, a]=ba-ab$ for all $b\in {\mathcal A}$,
is a derivation which is called the inner derivation induced by $a$.
A linear map $\delta: {\mathcal A}\rightarrow {\mathcal A}$ is called a generalized derivation of ${\mathcal A}$ if there exists a derivation
$d$ of ${\mathcal A}$ such that $\delta(ab) = a\delta(b) + d(a)b$ for all $a,b\in{\mathcal A}$. This coupling
is denoted by $(\delta, d)$.

A well-known theorem of Singer and Wermer  states that every bounded derivation on
a commutative Banach algebra has its image in the radical \cite{sing}.
About 30 years later, Thomas  extended the Singer-Wermer theorem to arbitrary, not necessarily
bounded, derivations \cite{thomas}. A number of authors have generalized this theorem in several
ways; see \cite{R77,mathieu,R1888,R21}. For example, a well-known result in \cite{mathieu} states that if $d$ is a bounded derivation of a Banach algebra ${\mathcal A}$ such that 
$[a, d(a)]\in Z({\mathcal A})$ for all $a\in{\mathcal A}$, then $d({\mathcal A})\subseteq{\rm rad}({\mathcal A})$.
Later in \cite{R1888} the authors prove the following:
If $d$ is a derivation of a Banach algebra ${\mathcal A}$ such that 
$[a, d(a)]\in Z({\mathcal A})$ for all $a\in{\mathcal A}$, then $d({\mathcal A})\subseteq{\rm rad}({\mathcal A})$.

Along these lines of research, Breasar and Mathieu \cite[Theorem 2.8]{R77} obtained a necessary and sufficient
condition for a derivation to be spectrally bounded on a unital
Banach algebra.

Throughout the paper ${\mathcal A}$ denotes a Banach algebra with the properties that
$\mathrm{rad}({\mathcal A})={\rm rann}({\mathcal A})$ and the algebra ${\mathcal A}/\mathrm{rad}({\mathcal A})$ is commutative. In harmonic analysis, some examples of (classes of) Banach algebras have been discovered
for which $\mathrm{rad}({\mathcal A})={\rm rann}({\mathcal A})$ and the algebra ${\mathcal A}/\mathrm{rad}({\mathcal A})$ is commutative. The most notable cases are $ L_0^\infty(G)^*$
when $G$ is an abelian locally compact group \cite{magh} and $VN(G)^*$
when $G$ is a discrete group \cite{lauff}.

As we shall see, the situation where the Banach algebra ${\mathcal A}$ satisfying
these conditions  provides a framework which allows us to
apply the well-known results concerning derivations and generalized derivations to  ${\mathcal A}$  and investigate the truth of all previous results
for ${\mathcal A}$. 

The organization of this paper is as follows. In Section 2, we consider an introverted subspace $X$ of $VN(G)$, the von
Neumann algebra generated by the left regular representation of $G$, with the following properties
\begin{eqnarray}\label{m}
	C_\lambda^*(G) \subseteq {X},\quad A(G) \cdot {X} \subseteq C_\lambda^*(G).
\end{eqnarray}
For the Banach algebra $X^*$,
equipped with the Arens product, we show that $\mathrm{rad}(X^*)={\rm rann}(X^*)$ and the algebra ${X^*}/\mathrm{rad}(X^*)$ is commutative. We also prove that $X^*$ is neither commutative nor
semiprime when $X\neq C^*_\lambda(G)$.

In Section 3, we first show that every derivation
of the Banach algebra ${\mathcal A}$ satisfying
the needed conditions, has its image in the radical. Using this, we prove that when ${\mathcal A}$ has a right identity every generalized derivation of ${\mathcal A}$ is spectrally bounded. We also show that a generalized derivation $(\delta, d)$ of ${\mathcal A}$ has its image in the radical if and only if $\delta=d$ and this holds if and only if $\delta$ is spectrally infinitesimal.

In Section 4, we investigate Posner's second theorem \cite{pos} and
show that the zero map is the only centralizing derivation of ${\mathcal A}$.
We characterize the space of all inner derivations of ${\mathcal A}$ and prove that ${\mathcal A}$ is commutative
if and only if any  derivation of ${\mathcal A}$ is zero. We also investigate our results for $k$-centralizing generalized derivations and prove similar results. For a generalized derivation  $(\delta, d)$ of ${\mathcal A}$ we obtain a necessary and sufficient
condition for $(\delta^2, d^2)$ to be still a
generalized derivation of ${\mathcal A}$.

In Section 5, the problem of $k$-skew centralizing generalized derivations of ${\mathcal A}$ are discussed.

We shall now fix some notation. Throughout this paper, $G$ denotes a locally compact group. Let $ B(G) $  denote the {Fourier-Stieltjes algebra}   of  $G$ consisting of all coefficient
functions  arising from all the weakly continuous unitary representations of $ G $.
 Then $B(G)$ can be
identified with the dual of $C^*(G)$, the group $C^*$-algebra of $G$. Also,
$B(G)$ with pointwise multiplication and the dual norm is a commutative semisimple Banach
algebra.


 Let  $P(G)$  be the set of all continuous positive definite functions on $G$. Let $P_\lambda(G)$ denote the closure of $P(G)\cap C_c(G)$ in the compact open topology,
where $C_c(G)$ is the set of all continuous functions on $G$ with compact support, and let
$B_\lambda(G)$ denote the linear span of $P_\lambda(G)$. Then $B_\lambda(G)$ is a closed ideal in $B(G)$ and
$B_\lambda(G)$ is precisely the dual of $C^*_\lambda(G)$, where $C^*_\lambda(G)$
is the norm closure of $\{\lambda(f): f\in L^1(G)\}$ in the operator algebra $ B(L^2(G))$, and $\lambda(f)(g) = f*g$ for all
$g\in L^2(G)$.

The Fourier algebra $A(G)$ is the norm-closed linear span of $P(G)\cap
C_c(G)$ in $B(G)$. Then $A(G)$ is a closed ideal in $B(G)$ and $A(G)\subseteq B_\lambda(G)$; see \cite[Page 208]{eym}. Furthermore, the dual of $ A(G) $ is isometrically isomorphic to $VN(G)$, the {group von Neumann algebra} which is generated by $\lambda$ in $ B(L^2(G))$. It is known from \cite[Theorem 1.4.1]{jadid2} that $B_\lambda(G)= B(G)$ if and only if $G$ is amenable, or equivalently $A(G)$ has a bounded approximate
identity. See \cite{eym} and \cite{kal} for more information on $ B(G) $, $ A(G)$ and $VN(G)$.

Let $\mathcal{A}$ be Banach algebra. Then $\mathcal{A}^*$ is a Banach $\mathcal{A}$-bimodule in a natural way. Let $ X $ be a norm closed $\mathcal{A}$-subbimodule of $\mathcal{A}^*$. For $ m\in X^* $ and $ x \in X $, define $ m \cdot x \in \mathcal{A}^* $ by
$$\langle m \cdot x , a\rangle = \langle m , x \cdot a\rangle \quad (a \in \mathcal{A}).$$
If $ m \cdot x \in X $ for all $ m \in X^* $ and $ x\in X $, then $ X $ is called a { (left) introverted } subspace of $\mathcal{A}^*$. The dual of an introverted subspace can be turned into a Banach algebra
if for all $ m,n \in X^* $ we define $ m \cdot n \in X^* $ by $ \langle m \cdot n , x\rangle = \langle m , n\cdot x \rangle$. In particular, we have the first Arens product on $ \mathcal{A}^{**} $ by taking $ X= \mathcal{A}^{*}$; see \cite{dal-lau} for more details. If $X$ is faithful (that is, $a = 0$ whenever $\langle x , a \rangle = 0$ for all $ x \in X $), then the natural map of $ \mathcal{A} $ into $ X^*$ is injective. The space $X^*$ can be equipped with the weak$^*$-topology
$\sigma(X^*,X)$. It is well known that for each $ m \in X^* $, the map $ m\mapsto m \cdot n, X^* \longrightarrow X^* $, is weak$^*$ continuous.
In addition to ${\mathcal A}^*$, another example of introverted subspace of ${\mathcal A}^*$ is the closed linear span of ${\mathcal A}^*\cdot {\mathcal A}$ in ${\mathcal A}^*$ which is denoted by $LUC({\mathcal A})$. If $G$ is a locally compact group,
then $LUC(A(G))$ coincides with $UCB(\widehat{G})$, the space of uniformly continuous functionals on $A(G)$; see \cite{granier}.

Note that even if ${\mathcal A}$ is commutative, $X^*$ may not be commutative
in general. For example,  $ L^\infty(G)^*$
for $G$ abelian and infinite \cite{civin} and $VN(G)^*$
for $G$ amenable and infinite \cite{lauff} are not commutative.

A linear map $T:{\mathcal A}\rightarrow {\mathcal A}$ is called a right (resp. left) multiplier of ${\mathcal A}$ if it satisfies $$ T(ab)= aT(b) \quad({\rm resp.}~T(ab) = T(a)b)\quad(a,b\in {\mathcal A}). $$
For each $ a \in \mathcal{A} $, let $ R_a: \mathcal{A}\rightarrow \mathcal{A} $ (resp. $ L_a: \mathcal{A}\rightarrow \mathcal{A}$)	be the multiplication map defined by $$R_a (b)= ba\quad({\rm resp.}~ L_a(b)= ab)\quad (b\in\mathcal{A}). $$ Then it is easy to see that $R_a$ (resp. $L_a$) is a right (resp. left) multiplier of ${\mathcal A}$.

\section{Remark and Examples}\label{s2}
Throughout this section, ${X}$ will denote an introverted subspace of $VN(G)$ satisfying (\ref{m}).
In the sequel we give some introverted subspaces of $VN(G)$ satisfying these conditions.	

\begin{example}\label{ex}
	{\rm	Let $G$ be a locally compact group.
		\begin{itemize}
			\item[(i)] It is known that $X=C_\lambda^*(G)$ is an introverted subspace of $VN(G)$ satisfying (\ref{m}); see \cite[Proposition 5.2]{lauf}.
			
			\item[(ii)] If $G$ is  discrete, then by \cite[Proposition 4.5]{lauf}, we have $UCB(\widehat{G})=C_\lambda^*(G)$ and so $X=VN(G)$ satisfying (\ref{m}).
		\end{itemize}
	}
\end{example}

Let $ \iota : C_{\lambda}^{*}(G) \longrightarrow {X} $ be the natural embedding. Then it is easy to see that $\pi:=\iota^* : {X}^* \rightarrow C_{\lambda}^{*}(G)^{*}$ is the natural restriction map. A simple computation shows
that $\pi$ is the identity on $ A(G) $ and it is a weak$^*$-weak$^*$ continuous algebra homomorphism when $ C_{\lambda}^{*}(G)^*$ is equipped with the induced
Arens product. Moreover, we recall that the Arens product on $ C_{\lambda}^{*}(G)^*= B_\lambda(G)$ is precisely the pointwise multiplication; see \cite[Proposition 5.3]{lauf}.


Before giving the following results, note that $X\cdot A(G)\subseteq C^*_\lambda(G)$. Hence, for every $n\in X^*$ and $x\in X$, we may define
the functional $\pi(n)\bullet x\in VN(G)$ as follows
$$
\langle \pi(n)\bullet x, \varphi\rangle=\langle \pi(n), x\cdot\varphi\rangle\quad(\varphi\in A(G)).
$$
It follows that $\pi(n)\bullet x=n\cdot x\in X$.
Thus, for every $m, n\in X^*$ we can define the functional
$m\bullet\pi(n)\in X^*$ as follows
$$
\langle m\bullet\pi(n), x\rangle=\langle m, \pi(n)\bullet x\rangle\quad(x\in X).
$$

\begin{lemma}\label{lem0}
	Let the map $ \pi : {X}^* \rightarrow B_{\lambda}(G)$ be the natural restriction map. Then the following statements hold.
	\begin{itemize}
		\item[{\rm(i)}] $ m \cdot n = m\bullet\pi(n) $ for all $ m, n \in {X}^* $.
		
		\item[{\rm(ii)}] $ \ker \pi = {\rm rann}( {X}^* )$.
		
		\item[{\rm(iii)}] $\mathrm{rad}(X^{*})={\rm rann}(X^{*})$.
	\end{itemize}
\end{lemma}
\begin{proof}
(i). Let $ m, n \in {X}^* $. Then there is a net $ (\varphi_\alpha) $ in $A(G) $ such that $\varphi_\alpha \rightarrow m $ in the weak$^*$-topology of $ {X}^* $. For each $ x \in {X}$, since $ x \cdot \varphi_\alpha \in C_{\lambda}^{*}(G) $ we have
\begin{eqnarray*}
	\langle m \cdot n , x \rangle &=& \lim_\alpha \langle \varphi_\alpha \cdot n , x \rangle = \lim_\alpha \langle n , x \cdot \varphi_\alpha \rangle\\
	&=&\lim_\alpha \langle \pi ( n) , x \cdot \varphi_\alpha \rangle
	=\lim_\alpha \langle \pi ( n) \bullet x , \varphi_\alpha \rangle\\
	&=& \langle m , \pi ( n) \bullet x \rangle=\langle m \bullet \pi ( n) , x \rangle.
\end{eqnarray*}
This shows that $ m \cdot n = m \bullet \pi(n)$.

(ii). Let $ n \in \ker \pi $ and $ m \in {X}^* $. Then there is a bounded net $ (\varphi_\alpha)$ in $ A(G) $ such that $\varphi_\alpha \rightarrow m $
in the weak$ ^*$-topology of $ {X}^*$. Thus, for each $ x \in {X} $ we have
$$\langle m \cdot n , x \rangle =\lim_\alpha \langle \pi (n) , x \cdot \varphi_\alpha \rangle = 0, $$
since $ x \cdot \varphi_\alpha \in {X} \cdot A(G) \subseteq C_{\lambda}^{*}(G)$ for all $ \alpha $. Therefore, $ \ker \pi \subseteq {\rm rann}( {X}^* ) $.
To prove the reverse inclusion, let $ r \in {\rm rann}( {X}^* ) $. Then $ m \cdot r =0 $ for all $ m \in {X}^* $. It follows that
$$\langle m , r \cdot x\rangle =\langle m \cdot r , x\rangle = 0 \qquad ( m \in {X}^*, x \in {X} ).$$
Therefore, $ r \cdot x =0 $ for all $x \in {X} $. Thus, for each $ \varphi \in A(G)$ and $ x \in C_{\lambda}^{*}(G)$ we have
$$\langle \pi(r) , x \cdot \varphi\rangle = \langle r \cdot x , \varphi \rangle=0.$$
Since $ C_{\lambda}^{*}(G)\cdot A(G)$ is dense in $ C_{\lambda}^{*}(G)$; see \cite[Proposition 4.4]{lauf}, we conclude that $ \pi(r) =0 $.

(iii). Since $\pi$ is an epimorphism from $X^{*}$ onto semisimple Banach algebra $B_\lambda(G)$, it follows from \cite[Chapter 25, Proposition 10]{bon} that
\begin{equation}
	\mathrm{rad}(X^{*})\subseteq \ker\pi={\rm rann}(X^{*}). \label{r1r}
\end{equation}
Moreover, ${\rm rann}(X^{*})$ is nil and must be included in $\mathrm{rad}(X^{*})$; see \cite[Proposition 1.5.6]{dales}. Hence, $\mathrm{rad}(X^{*})={\rm rann}(X^{*})$.
\end{proof}

As an immediate consequence of previous lemma we have the following result.
\begin{corollary}
	The quotient Banach algebra $X^*/\mathrm{rad}(X^{*})$ is isometrically isomorphic
	to $B_\lambda(G)$.
\end{corollary}

We recall that, since $C_\lambda^*(G) \subseteq {X}$, then $X$ is faithful and the natural map of $A(G)$ into $X^*$ is an embedding, and we will regard
$A(G)$ as a subalgebra of $X^*$. In fact, for each $\varphi \in A(G)$, we have
\begin{eqnarray*}
	\Vert \varphi \Vert_{A(G)}\geq\Vert \varphi\Vert_{ {X}^*}&=&
	\sup\{\vert \langle \varphi , x\rangle\vert: x\in X, \Vert x\Vert\leq 1\}\\
	&\geq&
	\sup\{\vert \langle \varphi , x\rangle\vert: x\in C^*_\lambda(G), \Vert x\Vert\leq 1\}\\
	&=& \Vert \varphi \Vert_{B_\lambda(G)}= \Vert \varphi \Vert_{A(G)};
\end{eqnarray*}
that is $\|\varphi\|_{A(G)}=\|\varphi\|_{X^*}$.

\begin{corollary}
	$A(G)$ is an ideal in $ {X}^*$.
\end{corollary}
\begin{proof}
Suppose that $\varphi \in A(G)$ and $m \in {X}^*$. Then
$m\cdot \varphi=\varphi\cdot m=\varphi\bullet\pi(m)=\varphi \pi(m)\in B_\lambda(G)$, since $A(G)$ is an ideal in $B_\lambda(G)$ and
$\pi(m)\in B_\lambda(G)$. It follows that $A(G)$ is an ideal in $ {X}^*$.
\end{proof}

If $G$ is an amenable locally compact group, then $A(G)$ has a bounded approximate identity $(\varphi_\alpha)$ such that $\|\varphi_\alpha\|=1$. It is easy to see that any weak$^*$-cluster point $E$
of $(\varphi_\alpha)$ is a right identity for $X^{*}$ with $\|E\|=1$.
\begin{corollary}\label{cor5}
	The following statements hold.
	\begin{itemize}
		\item[{\rm (i)}] $X^{*}$ has a right identity if and only if $G$ is amenable.
		
		\item[{\rm (ii)}] $X^{*}$ has a left identity if and only if $X^{*}=B(G)$.
	\end{itemize}
\end{corollary}
\begin{proof}
(i). Suppose that $G$ is amenable. Then $A(G)$ has a bounded approximate identity, say $(\varphi_\alpha)$. It is easy to see that any weak$^*$-cluster point $E$
of $(\varphi_\alpha)$ is a right identity of the Banach algebra $X^{*}$. For the converse, suppose that $ E $ is a right identity of $ {X}^*$. Then it is easy to see that $\pi(E)$ is the identity element of $B_\lambda(G)$. Therefore, $G$ is amenable.

(ii). Suppose that $ E $ is a left identity of $ {X}^*$. Then Lemma \ref{lem0}(ii) implies that $\ker\pi=\{0\}$. Thus, $\pi$ is an algebra isomorphism, that is, $X^{*}=B_\lambda(G)$. Now, let $ \mu \in B_\lambda(G) $ and $ m \in {X}^* $ be any Hahn-Banach extension of $ \mu $. Then it follows that
$$\mu \pi(E)=\pi(E)\mu = \pi(E)\cdot \mu = \pi(E)\cdot\pi (m) = \pi(E\cdot m) = \pi(m) = \mu.$$
Hence, $B_\lambda(G)$ is unital and so
$X^{*}=B_\lambda(G)=B(G)$. The converse, is trivial.
\end{proof}

An algebra $\mathcal{A}$ is called semiprime if for any $a\in{\mathcal A}$, $a{\mathcal A}a$ implies $a=0$. It is known
that every semisimple algebra is semiprime.
\begin{lemma}\label{lemma1}
	Let $G$ be amenable. Then the following statements are equivalent.
	\begin{itemize}
		\item[\rm (i)] $X^{*}$ is semisimple.
		
		\item[{\rm (ii)}] $X^{*}$ is semiprime.
		
		\item[{\rm (iii)}] $X=C_{\lambda}^{*}(G)$
		
		\item[{\rm (iv)}] $X^{*}$ is commutative.
		
		\item[{\rm (v)}] $X^{*}$ is unital.
	\end{itemize}
\end{lemma}
\begin{proof} The implications (i)$\Rightarrow$(ii),	(iii)$\Rightarrow$(iv) and (iv)$\Rightarrow$(v) are trivial. The implication (v)$\Rightarrow$(i) follows from Corollary \ref{cor5}(ii). It suffice to prove (ii)$\Rightarrow$(iii). Suppose that $X^{*}$ is semiprime. Then it is easy to see that $\ker\pi={\rm rann}(X^{*})=\{0\}$.
Now, by Hahn-Banach theorem we conclude that $X=C_{\lambda}^{*}(G)$.
\end{proof}

\section{Singer-Wermer Conjecture for Generalized Derivations}
Throughout the rest, ${\mathcal A}$ denotes a Banach algebra with the properties that
$\mathrm{rad}({\mathcal A})={\rm rann}({\mathcal A})$ and the algebra ${\mathcal A}/\mathrm{rad}({\mathcal A})$ is commutative.

\begin{remark}\label{rem1}

	{\rm (i) Let $G$ be an infinite discrete amenable group. Then it follows from \cite[Theorem 3]{gran} that $VN(G)\neq C_\lambda^*(G)=UCB(\widehat{G})$. Since $G$ is discrete, by Example \ref{ex}(ii), $VN(G)$ satisfies the conditions of (\ref{m}) and so $\mathrm{rad}(VN(G)^*)={\rm rann}(VN(G)^*)$ and the algebra ${VN(G)^*}/\mathrm{rad}(VN(G)^*)$ is commutative. Furthermore, the algebra $VN(G)^*$ has a right identity and it is never commutative by Lemma \ref{lemma1}.
	
	(ii) Let $G$ be an abelian  non-discrete locally compact  group. Then  $\mathrm{rad}(L_0^\infty(G)^*)={\rm rann}(L_0^\infty(G)^*)$ and the algebra ${L_0^\infty(G)^*}/\mathrm{rad}(L_0^\infty(G)^*)$ is commutative; see \cite{magh}. Furthermore, the algebra $L_0^\infty(G)^*$ has a right identity and it is never commutative.}
\end{remark}

\begin{theorem}\label{theorem1}
	Let $d$ be a derivation of ${\mathcal A}$. Then $d({\mathcal A})\subseteq \mathrm{rad}({\mathcal A})$.
\end{theorem}
\begin{proof}
Let $c\in \mathrm{rad}({\mathcal A})={\rm rann}({\mathcal A})$ and $a\in {\mathcal A}$. Then
$$0=d(ac)=ad(c)+d(a)c=ad(c).$$
This shows that $d(\mathrm{rad}({\mathcal A}))\subseteq \mathrm{rad}({\mathcal A})$.
It is easy to see that the map $\widetilde{d}: {\mathcal A}/\mathrm{rad}({\mathcal A})\rightarrow {\mathcal A}/\mathrm{rad}({\mathcal A})$ defined by $\widetilde{d}(a)=d(a)+\mathrm{rad}({\mathcal A})$ is a derivation. Moreover, by assumption ${\mathcal A}/\mathrm{rad}({\mathcal A})$ is a commutative semisimple Banach algebra. Therefore, by \cite{thomas} we obtain  $\widetilde{d}=0$. This implies that
$d({\mathcal A})\subseteq \mathrm{rad}({\mathcal A})$.\end{proof}

Before giving the following consequence of Theorem \ref{theorem1}, let us recall that a linear mapping $T:{\mathcal A}\rightarrow {\mathcal A}$ is called spectrally bounded if there exists $M\geq 0$ such that
$r(T(a))\leq M r(a)$ for all $a\in {\mathcal A}$, where $r(a)$ denotes the spectral radius of $a\in {\mathcal A}$. In
addition, if $M = 0$, $T$ is called spectrally infinitesimal.
\begin{corollary}\label{theorem2}
	The following statements hold.
	\begin{itemize}
		\item[{\rm (i)}]
		The composition of two derivations of ${\mathcal A}$ is always a derivation of ${\mathcal A}$.
		\item[{\rm (ii)}]
		Every derivation of ${\mathcal A}$ is spectrally infinitesimal.
	\end{itemize}
\end{corollary}
\begin{proof}
(i). This is an immediate consequence of Theorem \ref{theorem1}.

(ii). Let $d$ be a derivation on ${\mathcal A}$. Since $d({\mathcal A})\subseteq \mathrm{rad}({\mathcal A})$ and $\mathrm{rad}({\mathcal A})={\rm rann}({\mathcal A})$, it follows that $(d(a))^{i}=0$ for all $i\geq 2$ and $a\in {\mathcal A}$. Therefore, by spectral radius theorem, $r(d(a))=0$ and hence $d$ is spectrally infinitesimal.
\end{proof}

\begin{theorem}\label{theorem15}
	If ${\mathcal A}$ has a right identity, then every generalized derivation of ${\mathcal A}$ is spectrally bounded.
\end{theorem}
\begin{proof}
Let $(\delta,d)$ be a generalized derivation of ${\mathcal A}$. Since
the Banach algebra ${\mathcal A}/\mathrm{rad}({\mathcal A})$ is commutative, it follows from \cite{R25} that the spectral radius is subadditive and submultiplicative on ${\mathcal A}$. Suppose that $e\in{\mathcal A}$ is a right identity for ${\mathcal A}$. Then there are positive numbers $\alpha$ and $\beta$ such that
\begin{eqnarray*}
	r(\delta(a))&=&r(\delta(ae))\nonumber\\
	&=&r(a \delta(e)+d(a))\nonumber\\
	&\leq& \alpha(r(a\delta(e))+r(d(a))\nonumber\\
	&\leq& \alpha(\beta r(\delta(e))r(a)+r(d(a)),
\end{eqnarray*}
for all $a\in {\mathcal A}$. By Corollary \ref{theorem2} we have $r(d(a))=0$ for all $a\in {\mathcal A}$. Therefore, we obtain
$$r(\delta(a))\leq \alpha \beta r(\delta(e))r(a),$$
for all $a\in {\mathcal A}$. That is, $\delta$ is spectrally bounded.
\end{proof}

\begin{theorem}\label{theorem13}
	Let $(\delta,d)$ be a generalized derivation of ${\mathcal A}$. If ${\mathcal A}$ has a right identity, then the following statements are equivalent.
	\begin{itemize}
		\item[\rm(i)]
		$\delta({\mathcal A})\subseteq {\mathrm rad}({\mathcal A})$
		\item[\rm(ii)]
		$\delta$ is spectrally infinitesimal.
		\item[\rm(iii)]
		$\delta=d$.
	\end{itemize}
\end{theorem}
\begin{proof}
(i)$\Rightarrow$(ii). Suppose that $\delta({\mathcal A})\subseteq \mathrm{rad}({\mathcal A})$.
Since $\mathrm{rad}({\mathcal A})={\rm rann}({\mathcal A})$, we conclude that $\delta$ is spectrally infinitesimal.

(ii)$\Rightarrow$(iii).	First we note that the Banach algebra
${\mathcal A}/\mathrm{rad}({\mathcal A})$ is commutative. Thus, it follows from \cite{R25} that the spectral radius is submultiplicative on ${\mathcal A}$. Therefore, there exists $\alpha>0$ such that for each $a\in {\mathcal A}$ we have
$$r(a \delta(e))\leq \alpha r(a)r(\delta(e))=0,$$
where $e\in{\mathcal A}$ is a right identity for ${\mathcal A}$. Hence, by \cite[Proposition 1(ii)]{bon}, we have $\delta(e)\in \mathrm{rad}({\mathcal A})={\rm rann}({\mathcal A})$. Thus
$$\delta(a)=\delta(ae)=a \delta(e)+d(a)e=d(a),$$
for all $a\in {\mathcal A}$. This shows that $\delta=d$.
The implication (iii)$\Rightarrow$(i) follows from Theorem \ref{theorem1}.
\end{proof}

Let $T$ be a right multiplier of ${\mathcal A}$ and let $0$ denotes the zero derivation of ${\mathcal A}$. Then it is easy to see that $(T,0)$ is a generalized derivation of ${\mathcal A}$. Thus, we have the following result as an immediate consequence of Theorem \ref{theorem13}.

\begin{corollary}\label{corollary9}
	Let $T$ be a right multiplier of ${\mathcal A}$. If ${\mathcal A}$ has a right identity, then the following statements are equivalent.
	\begin{itemize}
		\item[\rm(i)]
		$T({\mathcal A})\subseteq \mathrm{rad}({\mathcal A})$.
		\item[\rm(ii)]
		$T$ is spectrally infinitesimal.
		\item[\rm(iii)]
		$T$ is a derivation on ${\mathcal A}$.
		\item[\rm(iv)]
		$T=0$.
	\end{itemize}
\end{corollary}

\section{$k$-Centralizing Generalized Derivations}
In this section, we investigate Posner's second theorem and
show that the zero map is the only centralizing derivation of ${\mathcal A}$.
We also investigate our results for centralizing generalized derivations and prove similar results.

We recall that the algebraic center of ${\mathcal A}$ is denoted by $Z({\mathcal A})$. Let $a,b\in {\mathcal A}$ and $k$ be a fixed positive integer. Set
$$[a,b]_{k}:=[[a,b]_{k-1},b],$$
where
$$[a,b]_{0}:=a,\qquad [a,b]:=ab-ba.$$
A mapping $T: {\mathcal A}\rightarrow {\mathcal A}$ is called $k$-centralizing (resp. $k$-commuting) if $[T(a),a^{k}]\in Z({\mathcal A})$ (resp. $[T(a),a^{k}]=0$) for all $a\in {\mathcal A}$. In the case $k=1$, $T$ is said to be centralizing (resp. commuting).

\begin{theorem}\label{theorem5}
	Let $d$ be a derivation of ${\mathcal A}$ and let $k$ be a positive integer. If ${\mathcal A}$ has a right identity, then the following statements are equivalent.
	\begin{itemize}
		\item[\rm(i)]
		$d=0$.
		\item[\rm(ii)]
		$d$ is $k$-centralizing.
		\item[\rm(iii)]
		$d$ is $k$-commuting.
	\end{itemize}
\end{theorem}
\begin{proof}
The implication (i)$\Rightarrow$(ii) is trivial.

(ii)$\Rightarrow$(iii). Suppose that (ii) holds.
By assumption and Theorem \ref{theorem1}, we obtain
$$[d(a),a^{k}]=d(a)a^{k}=d(a^{k+1})\in {\rm rann}({\mathcal A})\cap Z({\mathcal A}).$$
Thus, $[d(a),a^{k}]=0$.

(iii)$\Rightarrow$(i). Let $d$ be $k$-commuting and let $e$ be a right identity of ${\mathcal A}$. Then
\begin{equation}
	d(e)=[d(e),e]=[d(e),e^{k}]=0.\label{(1)}
\end{equation}
Given $c\in {\rm rann}({\mathcal A})$, we obtain $(c+e)=(c+e)^{k}$. Thus,
\begin{equation}
	d(c)=[d(c+e),(c+e)]=[d(c+e),(c+e)^{k}]=0.\label{(2)}
\end{equation}
From (\ref{(1)}) and (\ref{(2)}), we get
$$d(a)=d(e a)+d(a-e a)=d(e) a+d(a-ea)=0,$$
for all $a\in {\mathcal A}$.
\end{proof}

\begin{corollary}\label{corollary6}
	Let $d$ be a derivation of ${\mathcal A}$. If ${\mathcal A}$ has a right identity, then the following statements are equivalent.
	\begin{itemize}
		\item[\rm(i)]
		$d=0$.
		\item[\rm(ii)]
		$d$ is centralizing.
		\item[\rm(iii)]
		For every $k\in {\Bbb N}$, $d$ is $k$-centralizing.
		\item[\rm(iv)]
		There exists $k\in {\Bbb N}$ such that $d$ is $k$-centralizing.
		\item[\rm(v)]
		For every $k\in {\Bbb N}$, $d$ is $k$-commuting.
		\item[\rm(vi)]
		There exists $k\in {\Bbb N}$ such that $d$ is $k$-commuting.
	\end{itemize}
\end{corollary}
\begin{theorem}\label{theorem6}
	If ${\mathcal A}$ has a right identity, then the following statements are equivalents.
	\begin{itemize}
		
		\item[\rm(i)]
		${\mathcal A}$ has an identity.
		\item[\rm(ii)]
		Any derivation of ${\mathcal A}$ is zero.
		\item[\rm(iii)]
		Any inner derivation of ${\mathcal A}$ is zero.
		\item[\rm(iv)]
		${\mathcal A}$ is commutative.
	\end{itemize}
\end{theorem}
\begin{proof}
The implications (iii)$\Rightarrow$(iv) and (v)$\Rightarrow$(i) are obvious.

(i)$\Rightarrow$(ii).
Let $e$ be an identity for ${\mathcal A}$ and let $d$ be a derivation of ${\mathcal A}$. Then by Theorem \ref{theorem1}, we have $d(a)=ed(a)=0$ for all $a\in{\mathcal A}$.

(iv)$\Rightarrow$(v). Suppose that any inner derivation of ${\mathcal A}$ is zero. Then $d_a(b)=ba-ab=0$ for all $a,b\in {\mathcal A}$. This implies that ${\mathcal A}$ is commutative.
\end{proof}
\begin{lemma}\label{lemma3}
	Let $a,b,c\in {\mathcal A}$. Then $[a,b]\in {\rm rann}({\mathcal A})$ and consequently $cab=cba.$
\end{lemma}
\begin{proof}
Consider the inner derivation $d_{b}$ of ${\mathcal A}$. Then by Theorem \ref{theorem1}, $[a,b]=d_{b}(a)\in {\rm rann}({\mathcal A})$. Thus, $cab-cba=cd_{b}(a)=0$. 
It follows that $cab=cba$.
\end{proof}
\begin{theorem}\label{therorem7}
	Let $d$ be a derivation of ${\mathcal A}$. If ${\mathcal A}$ has a right identity, then the following statements are equivalent.
	\begin{itemize}
		\item[\rm(i)]
		$d$ is an inner derivation.
		\item[\rm(ii)]
		There exists $b_{0}\in {\mathcal A}$ such that for each $k\in {\Bbb N}$, the mapping $a\mapsto d(a)+b_{0} a$ is $k$-commuting.
		\item[\rm(iii)]
		There exist $b_{0}\in {\mathcal A}$ and $k\in {\Bbb N}$ such that the mapping $a\mapsto d(a)+b_{0} a$ is $k$-commuting.
		\item[\rm(iv)]
		There exist $b_{0}\in {\mathcal A}$ and $k\in {\Bbb N}$ such that the mapping $a\mapsto d(b)+b_{0} a$ is $k$-centralizing.
	\end{itemize}
\end{theorem}
\begin{proof}
(i)$\Rightarrow$(ii). Suppose that $d$ is inner. Then there exists $b_{0}\in {\mathcal A}$ such that  $d(a)=a b_{0}-b_{0} a$, for all $a\in {\mathcal A}$. By Lemma \ref{lemma3}, for each $k\in {\Bbb N}$, we have
\begin{eqnarray*}
	[d(a)+b_{0} a,a^{k}]&=&[a b_{0},a^{k}]\\
	&=& a b_{0} a^{k}-a^{k} a b_{0}\\
	&=& aa^{k}b_{0}-a^{k} a b_{0}\\
	&=& a^{k+1} b_{0}-a^{k+1} b_{0}=0,
\end{eqnarray*}
which implies (ii). \\
The implications (ii)$\Rightarrow$ (iii) $\Rightarrow$ (iv) are trivial.

Suppose now that (iv) holds. Consider the map $D:{\mathcal A}\rightarrow {\mathcal A}$ defined by $D(a):=d(a)-[a,b_{0}]$. It is easy to see that $D$ is a derivation of ${\mathcal A}$. Thus
$$[D(a),a^{k}]=D(a) a^{k}=[d(a)+b_{0}a,a^{k}]\in Z({\mathcal A}).$$
We now invoke Theorem \ref{theorem5} to conclude that $D=0$, which implies that $d$ is inner.
\end{proof}

We denote by $\mathrm{InnD}({\mathcal A})$ the space of all inner derivations of ${\mathcal A}$.  
\begin{theorem}\label{theorem8}
	Suppose that ${\mathcal A}$ has a right identity. Then ${\rm InnD}({\mathcal A})$ is continuously
	linearly isomorphic to ${{\mathcal A}}/{Z({\mathcal A})}$.
\end{theorem}
\begin{proof}
Consider the mapping $\Gamma$ from ${{\mathcal A}}/{Z({\mathcal A})}$ into ${\rm Inn}D({\mathcal A})$ defined by $\Gamma(a+Z({\mathcal A}))=d_{a}$. It is clear that $\Gamma$ is
onto. To check the injectivity, suppose that $\Gamma(a+Z({\mathcal A}))=0$ for some $a\in {\mathcal A}$. Then $d_{a}(b)=b a-a b=0$, for all $b\in {\mathcal A}$, which implies that $a\in Z({\mathcal A})$. Thus, $\Gamma$ is injective.
Now, let $a,b\in {\mathcal A}$ and $z\in Z({\mathcal A})$. Then
\begin{eqnarray*}
	\|d_{a}(b)\|&=&\| ba-ab\| \\
	&\leq& \| ba-zb\|+\| zb-ab\| \\
	&=&\| ba-bz\| +\| z b-a b\| \\
	& \leq& \| b\| \| a-z\| +\| z-a\| \| b\| \\
	&=& 2\| b\|\| a-z\|.
\end{eqnarray*}
This implies that
$$\| \Gamma(a+Z({\mathcal A}))\| =\| d_{a}\| \leq 2\| a-z\|\quad (a\in {\mathcal A}, z\in Z({\mathcal A})).$$
Therefore,
\begin{eqnarray*}
	\| \Gamma(a+Z({\mathcal A}))\|\leq 2\| a+Z({\mathcal A})\|.
\end{eqnarray*}
This shows that the map is continuous.
\end{proof}

\begin{theorem}\label{theorem8}
	Let $(\delta,d)$ be a generalized derivation of ${\mathcal A}$ and $k\in {\Bbb N}$. If ${\mathcal A}$ has a right identity, then  the following statements are equivalent.
	\begin{itemize}
		\item[\rm(i)]
		$\delta$ is $k$-commuting.
		\item[\rm(ii)]
		$\delta$ is $k$-centralizing.
		\item[\rm(iii)]
		$\delta$ is a right multiplier.
		\item[\rm(iv)]
		There exists $b\in {\mathcal A}$ such that $\delta=R_b$.
	\end{itemize}
\end{theorem}
\begin{proof}
(i)$\Rightarrow$(ii). This is trivial.

(ii) $\Rightarrow$ (i). By induction and paying attention to the fact that $[a,b]\in {\rm rann}({\mathcal A})$ for all $a,b\in {\mathcal A}$, we have that
\begin{equation}
	[[\delta(a),a]_{i-1},a]=[\delta(a),a^{i}]\quad (a\in {\mathcal A}, i\in {\Bbb N}). \label{f1}
\end{equation}
Hence,
$$[[\delta(a),a]_{k-1},a]=[\delta(a),a^{k}]\in Z({\mathcal A})\quad (a\in {\mathcal A}).$$
This implies that the mapping $a\longmapsto [\delta(a),a]_{k-1}$
is centralizing. Using (\ref{f1}), we obtain
$$[\delta(a),a^{k}]=[[\delta(a),a]_{k-1},a]\in Z({\mathcal A})\cap {\rm rann}({\mathcal A})\quad (a\in {\mathcal A}).$$
This implies that $[\delta(a),a^{k}]=0$ for all $a\in {\mathcal A}$. Therefore, $\delta$ is $k$-commuting.

(i)$\Rightarrow$(iii).
Let $e$ be a right identity of ${\mathcal A}$ and $c\in {\rm rann}({\mathcal A})$. Since $(e+c)^{i}=e+c$ for all $i\in {\Bbb N}$, it follows that
\begin{eqnarray*}
	\delta(e+c)&=&\delta(e+c) (e+c)^{k}=(e+c)^{k} \delta(e+c)\\
	&=&\delta((e+c)^{k+1})-d((e+c)^{k}) (e+c)\\
	&=&\delta(e+c)-d(e+c).
\end{eqnarray*}
This shows that $d(c)=-d(e)$. Furthermore,
$$\delta (e)=e \delta(e)+d(e)=\delta(e) e+d(e)=\delta(e)+d(e).$$
Thus $d(e)=0$. Hence $d(c)=0$, which implies that $d(a-e a)=0$ for all $a\in {\mathcal A}$. Thus
$$d(a)=d(e a)=d(e) a+e d(a)=0\quad (a\in {\mathcal A}).$$
This implies that
$$\delta(a b)=a \delta(b)+d(a) b=a \delta(b)\quad (a,b\in {\mathcal A});$$
that is $\delta$ is a right multiplier.

(iii)$\Rightarrow$(iv).
Let $\delta$ be a right multiplier and let $e$ be a right identity of ${\mathcal A}$. Then it is straightforward to show that $\delta=R_{\delta(e)}$.

(iv) $\Rightarrow$(i).
Using Lemma \ref{lemma3}, we can see that
\begin{eqnarray*}
	[\delta(a),a^{k}]&=&\delta(a) a^{k}-a^{k} \delta(a)\\
	&=& a b a^{k}-a^{k} a b\\
	&=& a^{k+1} b-a^{k+1} b=0\quad (a\in {\mathcal A}).
\end{eqnarray*}
Thus, $\delta$ is $k$-commuting.
\end{proof}

As an immediate consequence of Theorem \ref{theorem8}, we have the following result.
\begin{corollary}\label{corollary7}
	Let $(\delta,d)$ be a generalized derivation of ${\mathcal A}$. If ${\mathcal A}$ has a right identity, then  the following statements are equivalent.
	\begin{itemize}
		\item[\rm(i)]
		$\delta$ is centralizing.
		\item[\rm(ii)]
		For every $k\in {\Bbb N}$, $\delta$ is $k$-centralizing.
		\item[\rm(iii)]
		There exists $k\in {\Bbb N}$ such that $\delta$ is $k$-centralizing.
		\item[\rm(iv)]
		For every $k\in {\Bbb N}$, $\delta$ is $k$-commuting.
		\item[\rm(v)]
		There exists $k\in {\Bbb N}$ such that $\delta$ is $k$-commuting.
	\end{itemize}
\end{corollary}

By ${\frak C}({\mathcal A})$ we denote the set of all centralizing generalized derivations of ${\mathcal A}$. By Theorem \ref{theorem8}, we conclude that the product of two elements of ${\frak C}({\mathcal A})$ is always a centralizing generalized derivation of ${\mathcal A}$. Hence the space ${\frak C}({\mathcal A})$ is a Banach algebra as a subalgebra of $B({\mathcal A})$, the space of all bounded linear maps on ${\mathcal A}$.
\begin{theorem}\label{theorem9}
	Suppose that ${\mathcal A}$ has a right identity.	Then the Banach algebra ${\frak C}({\mathcal A})$ is  isomorphic to ${\mathcal A}/\mathrm{rad}({\mathcal A})$.
\end{theorem}
\begin{proof}
Let $e$ be right identity of ${\mathcal A}$. One can introduce a bounded linear map $\Gamma: {\frak C}({\mathcal A})\longrightarrow {\mathcal A}/\mathrm{rad}({\mathcal A})$ by $\Gamma(\delta):=\pi(\delta(e))$, where $\pi: {\mathcal A}\rightarrow{\mathcal A}/\mathrm{rad}({\mathcal A})$ is the natural homomorphism. Using Theorem \ref{theorem8}(iv) and Lemma \ref{lemma3}, we conclude that $\delta_{1}\circ \delta_{2}=\delta_{2}\circ \delta_{1}$ for all $\delta_{1},\delta_{2}\in {\frak C}({\mathcal A})$. Hence, ${\frak C}({\mathcal A})$ is a commutative Banach algebra. Therefore, for every $\delta_{1},\delta_{2}\in {\frak C}({\mathcal A})$, we have
\begin{eqnarray*}
	\Gamma(\delta_{1}\circ \delta_{2})&=\pi(\delta_{2}(\delta_{1}(e) e))
	=\pi(\delta_{1}(e) \delta_{2}(e))\\&=\pi(\delta_{1}(e)) \pi(\delta_{2}(e))
	=\Gamma (\delta_{1}) \Gamma (\delta_{2}).
\end{eqnarray*}
Hence, $\Gamma$ is a homomorphism. 
Now we show that $\Gamma$ is onto. Given $\pi(b)=b+\mathrm{rad}({\mathcal A})\in {\mathcal A}/\mathrm{rad}({\mathcal A})$, define $\delta \in {\frak C}({\mathcal A})$ by $\delta(a)=a b$. Then
\begin{eqnarray*}
	\Gamma(\delta)&=&\pi(\delta(e))=\pi(e b)\\
	&=&\pi(e)\pi(b)=\pi(b)\pi(e)\\&=&\pi(be)
	=\pi(b).
\end{eqnarray*}
This proves theorem.
\end{proof}

The maps $T$ and $S$ from ${\mathcal A}$ into ${\mathcal A}$ are called orthogonal, denoted by $T\perp S$, if $T(a)c S(b)=S(b) c T(a)=0$ for all $a,b,c\in {\mathcal A}$.

\begin{theorem}\label{theorem10}
	Let $(\delta,d)$ be a generalized derivation of ${\mathcal A}$. If ${\mathcal A}$ has a right identity, then the following statements are equivalent.
	\begin{itemize}
		\item[\rm(i)]
		$[[\delta(a),a],\delta(a)]\in Z({\mathcal A})$ for all $a\in {\mathcal A}$.
		\item[\rm(ii)]
		$d\perp \delta$.
		\item[\rm(iii)]
		$(\delta^{2},d^{2})$ is a generalized derivation of ${\mathcal A}$.
	\end{itemize}
\end{theorem}
\begin{proof}
(i) $\Rightarrow$ (ii). Given $a\in {\mathcal A}$, by assumption and Lemma \ref{lemma3}, we have
$$[[\delta(a),a],\delta(a)]=0.$$
Moreover, it is easy to see that $\delta(a)=a \delta(e)+d(a)$, where $e$ is a right identity for ${\mathcal A}$. Now, Theorem \ref{theorem1} and Lemma \ref{lemma3} imply that
\begin{equation}
	[[\delta(a),a],\delta(a)]=d(a) a^{2} \delta(e).\label{r0}
\end{equation}
Thus,
\begin{equation}
	d(a)a^{2}\delta(e)=0.\label{r1}
\end{equation}
Replacing $a$ by $e$ in (\ref{r1}), we get
\begin{equation}
	d(e) \delta(e)=0.\label{r2}
\end{equation}
Replacing $a$ by $a+e$ in (\ref{r1}) and using (\ref{r2}), we obtain
\begin{equation}
	d(a) \delta(e)+2d(a) a \delta(e)=0.\label{r3}
\end{equation}
Putting $-a$ for $a$ in the relation above we have
\begin{equation}
	-d(a)\delta(e)+2d(a) a \delta(e)=0.\label{r4}
\end{equation}
It follows from (\ref{r3}) and (\ref{r4}) that $d(a) \delta(e)=0$. Therefore,
$$d(a)\delta(b)=d(a) b\delta(e)=d(a) \delta(e) b=0,$$
for all $a,b\in {\mathcal A}$. This implies that
\begin{equation}
	d(a)c\delta(b)=d(a)\delta(b)c=0,\label{r5}
\end{equation}
for all $a,b,c\in {\mathcal A}$.
Since $d({\mathcal A})\subseteq{\rm rann}({\mathcal A})$, it follows that
\begin{equation}
	\delta(b)cd(a)=0,\label{r6}
\end{equation}
for all $a,b,c\in {\mathcal A}.$
Thus, by (\ref{r5}) and (\ref{r6}) the statement (ii) holds.

(ii) $\Rightarrow$ (iii).
Suppose that $d$ and $\delta$ are orthogonal. Then for every $a,b\in {\mathcal A}$, we have
\begin{eqnarray*}
	\delta^{2}(a b)&=&\delta(\delta(ab))=\delta(a \delta(b)+d(a) b)\\
	&=&a \delta^{2}(b)+2d(a) \delta(b)+d^{2}(a) b\\
	&=&a \delta^{2}(b)+d^{2}(a) b.
\end{eqnarray*}
This shows that $(\delta^{2},d^{2})$ is a generalized derivation.

(iii) $\Rightarrow$ (i). Assume that $(\delta^{2},d^{2})$ is a generalized derivation of ${\mathcal A}$. Then
\begin{equation}
	\delta^{2}(ab)=a \delta^{2}(b)+d^{2}(a) b, \label{r7}
\end{equation}
for all $ a,b\in {\mathcal A}$. Moreover,
\begin{eqnarray}
	\delta^{2}(a b)&=&\delta(\delta(ab))=\delta(a \delta(b)+d(a) b)\nonumber\\
	&=& a \delta^{2}(b)+d(a)\delta(b)+d(a) \delta(b)+d^{2}(a) b \nonumber\\
	&=& a \delta^{2}(b)+2d(a) \delta(b)+d^{2}(a) b,\label{r8}
\end{eqnarray}
for all $ a,b\in {\mathcal A}$.
Comparing (\ref{r7}) and (\ref{r8}), we get that
\begin{equation}
	d(a)\delta(b)=0, \label{r9}
\end{equation}
for all $a,b\in {\mathcal A}$. Using (\ref{r0}) and (\ref{r9}), we obtain
\begin{eqnarray*}
	[[\delta(a),a],\delta(a)]&=& d(a) a^{2} \delta(e)\\
	&=& d(a) \delta(e) a^{2}=0,
\end{eqnarray*}
for all $a\in {\mathcal A}$ and therefore (i) holds.
\end{proof}

Obviously the product of two right multipliers is a right multiplier, and the product of two derivations of ${\mathcal A}$ is a derivation. The next example shows that this fact does not hold for all generalized derivations of ${\mathcal A}$.
\begin{example}
	{\rm Let $G={\Bbb Z}$. Then $G$ is an infinite discrete amenable group. Thus, 
		$VN(G)^*$ is a  non-commutative Banach algebra with a right identity, say $e$. It is easy to see that $(\delta,d)$ defined as $d=d_e$ and $\delta=R_e+d_e$ is a generalized derivation of $VN(G)^*$.
		Now, by Lemma \ref{lemma1}, there exists $c\in \mathrm{rad}(VN(G)^*)={\rm rann}(VN(G)^*)$ such that $c\neq 0$. It follows that
		\begin{equation}
			\delta^{2}(c)=4c,\quad c \delta^{2}(e)+d^{2}(c) e=2c.\label{rr1}
		\end{equation}
		But $(\delta^{2},d^{2})$ fails to be a generalized derivation. Indeed, if $(\delta^{2},d^{2})$ is a generalized derivation, then by (\ref{rr1}) we deduce that $c=0$ which is a contradiction.}
\end{example}
\section{ $k$-Skew Centralizing Generalized Derivations}
Let $a,b\in {\mathcal A}$ and $k$ be a fixed positive integer. Set
$$\langle a,b\rangle_{k}=\langle \langle a,b\rangle_{k-1},b\rangle,$$
where
$$\langle a,b\rangle_{0}=a,\quad \langle a,b\rangle =a b+b a.$$
A mapping $T: {\mathcal A}\rightarrow {\mathcal A}$ is called $k$-skew centralizing (resp. $k$-skew commuting) if it satisfies
$$\langle T(a),a^{k}\rangle \in Z({\mathcal A})\quad (\hbox{resp.}~ \langle T(a),a^{k}\rangle=0)\quad (a\in {\mathcal A}).$$
In the case $k=1$, $T$ is called skew centralizing (resp. skew commuting).

\begin{theorem}\label{theorem11}
	Let $(\delta,d)$ be a generalized derivation of ${\mathcal A}$ and $k\in {\Bbb N}$. If ${\mathcal A}$ has a right identity, then the following statements hold.
	\begin{itemize}
		\item[\rm(i)]
		If $\delta$ is $k$-skew centralizing, then there exists $z\in Z({\mathcal A})$ such that $\delta=R_z$.
		\item[\rm(ii)]
		If $\delta$ is $k$-skew commuting, then $\delta=0$.
	\end{itemize}
\end{theorem}
\begin{proof}
(i). By induction, we obtain
\begin{equation}
	[\langle \delta(a),a^{i}\rangle ,a]=[\delta(a),a^{i+1}], \label{f1f}
\end{equation}
for all $ a\in {\mathcal A}$ and $i\in {\Bbb N}$.
Now suppose that $\delta$ is $k$-skew centralizing. Then $\langle \delta(a),a^{k}\rangle \in Z({\mathcal A})$ for all $a\in {\mathcal A}$. Using (\ref{f1f}), we have
$$0=[\langle \delta(a),a^{k}\rangle,a]=[\delta(a),a^{k+1}],$$
for all $a\in {\mathcal A}$. This shows that $\delta$ is $(k+1)$-commuting. It follows from Theorem \ref{theorem8}, that there exists $b\in {\mathcal A}$ such that $\delta=R_b$.
Let $e$ be a right identity of ${\mathcal A}$. Then
$$2e b=\langle \delta(e),e^{k}\rangle \in Z({\mathcal A}).$$
This implies that
$$\delta(a)=a b=a(e b)=R_{e b}(a),$$
for all $ a\in {\mathcal A}$, which completes the proof of (i).

(ii). Let $\delta$ be a $k$-skew commuting generalized derivation of ${\mathcal A}$. By (i), there exists $z\in Z({\mathcal A})$ such that $\delta=R_z$. If $e$ is a right identity of ${\mathcal A}$, then
\begin{eqnarray*}
	0&=&\langle \delta(e),e^{k}\rangle =\langle e z,e\rangle =e z e+e ez\\
	&=&2e z=2z e=2z.
\end{eqnarray*}
Therefore, $\delta(a)=az=0$ for all $a\in {\mathcal A}$.
\end{proof}

\begin{lemma}\label{lemma*}
	$Z({\mathcal A})$ is a closed ideal of ${\mathcal A}$.
\end{lemma}
\begin{proof}
It is easy to check that $Z({\mathcal A})$ is a closed subspace of ${\mathcal A}$. To complete the proof, we show that $Z({\mathcal A})$ is an ideal in ${\mathcal A}$. By lemma \ref{lemma3}, we obtain
$$(z b) a=z a b=a(z b),$$
for all $z\in Z({\mathcal A})$ and $a,b\in {\mathcal A}$.
Similarly,
$$(b z)a =z b a=z a b=a z b=a(b z),$$
for all $z\in Z({\mathcal A})$ and $a,b\in {\mathcal A}$, which implies that $Z({\mathcal A})$ is an ideal in ${\mathcal A}$.
\end{proof}

\begin{corollary}\label{corollary**}
	Let $(\delta,d)$ be a generalized derivation of ${\mathcal A}$. If ${\mathcal A}$ has a right identity, then the following statements hold.
	\begin{itemize}
		\item[\rm(i)]
		$\delta$ is skew centralizing.
		\item[\rm(ii)]
		For every $k\in {\Bbb N}$, $\delta$ is $k$-skew centralizing.
		\item[\rm(iii)]
		There exists $k\in {\Bbb N}$ such that $\delta$ is $k$-skew centralizing.
	\end{itemize}
\end{corollary}
\begin{proof}
(i) $\Rightarrow$ (ii)
By Theorem \ref{theorem11}, there exists $z\in Z({\mathcal A})$ such that $\delta(a)=a z$ for all $a\in {\mathcal A}$. Hence by Lemma \ref{lemma*}, we have
$$\langle \delta(a),a^{k}\rangle =\langle a z,a^{k}\rangle =a za^{k}+a^{k} a z\in Z({\mathcal A}),$$
for all $a\in {\mathcal A}$, $k\in {\Bbb N}$.

(ii) $\Rightarrow$ (iii).
This is obvious.

(iii) $\Rightarrow$ (i).
By Theorem \ref{theorem11}, there exists $z\in Z({\mathcal A})$ such that $\delta(a)=a z$ for all $a\in {\mathcal A}$. Therefore by Lemma \ref{lemma*}, we have
$$\langle \delta(a),a\rangle=\langle a z,a\rangle =a z a+a a z\in Z({\mathcal A}),$$
for all $a\in {\mathcal A}$.
\end{proof}

By ${S\frak{C}}({\mathcal A})$ we denote the set of all skew centralizing generalized derivations on ${\mathcal A}$. By Theorem \ref{theorem11} and Lemma \ref{lemma*} we conclude that the space $S\frak{C}({\mathcal A})$ is a closed ideal of $\frak{C}({\mathcal A})$, the space of all centralizing generalized derivations of ${\mathcal A}$.
\begin{corollary}\label{corollary***}
	Suppose that ${\mathcal A}$ has a right identity.	Then the Banach algebra $S{\frak C}({\mathcal A})$ is  isomorphic to $Z({\mathcal A})$.
\end{corollary}
\begin{proof} We show that the mapping $\Gamma:S{\frak C}({\mathcal A})\longrightarrow Z({\mathcal A})$ defined by $\Gamma(\delta):=z$, where $\delta=R_z$ for some $z\in Z({\mathcal A})$, is an  isomorphism. To prove this, let $ \delta_{1},\delta_{2} \in S{\frak C}({\mathcal A})$ such that $\delta_{1}=\delta_{2}$. Then by Theorem \ref{theorem11} there exist $z_{1},z_{2}\in Z({\mathcal A})$ such that $\delta_{1}=R_{z_1}$ and  $\delta_{2}=R_{z_2}$. Now, suppose that $e$ is a right identity of ${\mathcal A}$. Then $z_1=\delta_{1}(e)=\delta_{2}(e)=z_2$, whence $\Gamma$ is well-defined. One can easily check that $\Gamma$ is an isomorphism. 
\end{proof}

Before giving the following result, recall that $(L_{b},d_{-b})$ is a generalized derivation of ${\mathcal A}$ for all $b\in {\mathcal A}$.
\begin{theorem}\label{theorem12}
	Let $z\in {\mathcal A}$ and ${\mathcal A}$ has a right identity. Then the following statements are equivalent.
	\begin{itemize}
		\item[\rm(i)]
		$L_{z}$ is centralizing.
		\item[\rm(ii)]
		$L_{z}$ is skew centralizing.
		\item[\rm(iii)]
		$z\in Z({\mathcal A})$.
	\end{itemize}
\end{theorem}
\begin{proof}
(i)$\Rightarrow$(iii). By assumption and Theorem \ref{theorem8} we obtain that $L_{z} $ is a right multiplier. Thus
$$za=L_{z}(a)=L_{z}(a e)=aL_{z}(e)=a (z e)=az,$$
for all $a\in {\mathcal A}$. This implies $z\in Z({\mathcal A})$.

(ii) $\Rightarrow$ (iii).
Suppose that $L_{z}$ is skew centralizing. Then by Theorem \ref{theorem11}, there exists $z'\in Z({\mathcal A})$ such that $L_{z}(a)=az'$ for all $a\in {\mathcal A}$. Thus
$$z=ze=L_{z}(e)=e z'=z' e=z'\in Z({\mathcal A}).$$

(iii) $\Rightarrow$ (i).
Let $z\in Z({\mathcal A})$. Then
\begin{eqnarray*}
	[L_{z}(a),a]&=&[z a,a]=za a-a z a\\
	&=&z a a-za a=0\in Z({\mathcal A}),
\end{eqnarray*}
for all $a\in {\mathcal A}$.

(iii) $\Rightarrow$ (ii).
Let $z\in Z({\mathcal A})$. By Lemma \ref{lemma*}, we have
$za^{2}\in Z({\mathcal A})$ for all $a\in {\mathcal A}$. Thus
\[
\langle L_{z}(a),a\rangle =2z a^{2}\in Z({\mathcal A}),
\]
for all $a\in {\mathcal A}$.
\end{proof}

\begin{theorem}\label{theorem14}
	Let $T$ be a left multiplier of ${\mathcal A}$. Then $T({\mathcal A})\subseteq\mathrm{rad}({\mathcal A})$ if and only if $T$ is spectrally infinitesimal.
\end{theorem}
\begin{proof} Suppose that $T({\mathcal A})\subseteq\mathrm{rad}({\mathcal A})$. By \cite[Proposition 1(i)]{bon}, we have
$r(T(a))=0$ for all $a\in {\mathcal A}$. This shows that $T$ is spectrally infinitesimal. Conversely, assume that
$T$ is spectrally infinitesimal. Then
$$r(T(a) b)=r(T(a b))=0,$$
for all $a,b\in {\mathcal A}$. Hence, \cite[Proposition 1(ii)]{bon} implies that $T(a)\in \mathrm{rad}({\mathcal A})$ for all $a\in {\mathcal A}$.
\end{proof}

By Theorem \ref{theorem12} we deduce that, the left multiplier $L_{z}$ is both centralizing and skew centralizing on ${\mathcal A}$ for all $z\in Z({\mathcal A})$. But this assertion is not true for all $z\in {\mathcal A}$ as shown in the following example.

\begin{example}\label{corollary8}
	{\rm  Let ${\mathcal A}=VN({\Bbb Z})^*$. Then there exists an element $z\in {\mathcal A}$ such that $L_{z}$ is neither centralizing nor skew centralizing. In fact, by Lemma \ref{lemma1} there exists $z\in {\rm rann}({\mathcal A})$ such that $z\neq 0$. If $L_{z}$ is either centralizing or skew centralizing, then $z\in Z({\mathcal A})$, by Theorem \ref{theorem12}. Now let $e$ be a right identity of ${\mathcal A}$. Then
		$z=z e=e z=0$. This contradiction will show that $L_z$ is neither centralizing nor skew centralizing.}
\end{example}


\bibliographystyle{amsplain}

\end{document}